\documentclass[12pt]{amsart}
\usepackage{amssymb,latexsym,amsmath,amscd,mathrsfs,yfonts,hyperref}
\usepackage{fullpage}
\input xy
\xyoption{all}

\newcommand{\ZZ}{\mathbb Z}

\newcommand{\CC}{\mathbb C}

\newcommand{\NN}{\mathbb N}

\newcommand{\cpt}{\mathbb K}

\def\C{\textup{C}}

\def\Ext{\textup{Ext}}

\def\Hom{\textup{Hom}}

\def\SW{\mathtt{SW}}

\def\Mod{\mathtt{Mod}}

\def\prot{\hat{\otimes}}

\def\NSH{\mathtt{NSH}}
\def\NS{\mathtt{NSH}^f}

\def\K{\textup{K}}
\def\SS{\mathbb{S}}

\def\KK{\textup{KK}}
\def\E{\textup{E}}

\def\1{\bf{1}}

\def\Csep{\mathtt{SC^*}}

\def\dlim{\varinjlim}

\newcommand{\map}{\rightarrow}
\newcommand{\functor}{\longrightarrow}

\newcommand{\beq}{\begin{eqnarray}}
\newcommand{\beqn}{\begin{eqnarray*}}
\newcommand{\eeq}{\end{eqnarray}}
\newcommand{\eeqn}{\end{eqnarray*}}

\theoremstyle{definition}
\newtheorem{thm}{Theorem}[section]
\theoremstyle{definition}
\newtheorem*{Thm}{Theorem}
\newtheorem*{Rem}{Remark}

\newtheorem*{Conj}{Conjecture}

\newtheorem{lem}[thm]{Lemma}
\newtheorem{prop}[thm]{Proposition}

\newtheorem{defn}[thm]{Definition}
\newtheorem{rem}[thm]{Remark}
\newtheorem{que}[thm]{Question}

\begin{document}

\title{On the Generating Hypothesis in noncommutative stable homotopy}
\author{Snigdhayan Mahanta}
\email{snigdhayan.mahanta@mathematik.uni-regensburg.de}
\address{Fakult{\"a}t f{\"u}r Mathematik, Universit{\"a}t Regensburg, 93040 Regensburg, Germany.}
\thanks{This research was supported by the Deutsche Forschungsgemeinschaft (SFB 878), ERC through AdG 267079, and the Humboldt Professorship of M. Weiss.}

\begin{abstract}
Freyd's Generating Hypothesis is an important problem in topology with deep structural consequences for finite stable homotopy. Due to its complexity some recent work has examined analogous questions in various other triangulated categories. In this short note we analyze the question in noncommutative stable homotopy, which is a canonical generalization of finite stable homotopy. Along the way we also discuss Spanier--Whitehead duality in this extended setup.
\end{abstract}

\maketitle

\begin{center}
{\bf Introduction}
\end{center}

In (finite) stable homotopy theory the Spanier--Whitehead category of finite spectra, denoted by $\SW^f$, is a central object of study. Roughly speaking, it is constructed by formally inverting the suspension functor on the category of finite pointed CW complexes and it is a triangulated category. Its counterpart in the noncommutative setting is the triangulated noncommutative stable homotopy category, denoted by $\NSH$. This triangulated category was constructed by Thom \cite{ThomThesis} (see also \cite{Ambrogio}) building upon earlier works of Rosenberg \cite{RosNCT}, Schochet \cite{SchTopMet2}, Connes--Higson \cite{ConHig}, D\u{a}d\u{a}rlat \cite{DadAsymHom} and Houghton-Larsen--Thomsen \cite{HouTho} amongst others. The triangulated category $\NSH$ is a canonical generalization of $\SW^f$. It comes in a mysterious package carrying vital information about bivariant homology theories on the category of separable $C^*$-algebras. Triangulated category structures also play an important role in the theory of $C^*$-algebras; for instance, the work of Meyer--Nest on the Baum--Connes conjecture via localization of triangulated categories has had tremendous impact \cite{MeyNes}.

\smallskip
\noindent
Freyd stated the following {\em Generating Hypothesis} in \cite{FreGH} (Chapter 9): 

\begin{Conj}[Freyd]
The object $((S^0,\star),0)$ is a graded generator in $\SW^f$, where $(S^0,\star)$ is the pointed $0$-sphere ($\star$ being the basepoint).
\end{Conj}

An alternative formulation of the Generating Hypothesis asserts that for any two finite spectra $X,Y$ the canonical homomorphism 
\beq \label{GH}
\Phi: \SW^f(X,Y)\map\Hom_{\pi_*(\SS)}(\pi_*(X),\pi_*(Y))
\eeq is injective. Here $\SS$ denotes the sphere spectrum and $\Mod(\pi_*(\SS))$ denotes the category of right modules over the graded commutative ring $\pi_*(\SS)$. The conjecture has some interesting reformulations and generalizations \cite{FreGH,DevGH,HovGH,BohMay,SheStr}. It remains an open problem at the time of writing this article. However, it has spurred a lot of stimulating research. Analogues of the Generating Hypothesis have been addressed in several other contexts, such as the stable module category of a finite group algebra \cite{BenCheChrMin,CarCheMin}, the derived category of a ring \cite{Lockridge,HovLocPun}, equivariant stable homotopy \cite{Bohmann}, and so on. By Proposition 9.7 of \cite{FreIdem} (see also Corollary 3.2 of \cite{HovGH}) the injectivity of the map $\Phi$ in $\SW^f$ automatically implies its bijectivity. If true, the Generating Hypothesis would reduce the task of understanding the stable homotopy classes of maps between finite pointed CW complexes to a more tractable algebraic problem, i.e., understanding the module category of $\pi_*(\SS)$. 

A straightforward generalization of the Generating Hypothesis to the noncommutative setting would predict that the canonical map $\Phi':\NSH(A,B)\map\Hom_{\pi_*(\CC)}(\pi_*(A),\pi_*(B))$  in $\NSH$ is injective. We call it the {\em Na{\"i}ve Generating Hypothesis} in $\NSH$. This question is motivated by the algebraization problem of noncommutative stable homotopy. This assertion generalizes the following {\em Cogenerating Hypothesis} in finite stable homotopy: The canonical map $\Phi:\SW^f(Y,X)\map\Hom_{\pi^*(\SS)}(\pi^*(X),\pi^*(Y))$ is injective, where $\pi^*(-)$ denotes the stable cohomotopy functor. Observe that $(\SW^f)^\textup{op}$ sits inside $\NSH$ as a full triangulated subcategory. In $\SW^f$ there is a contravariant duality functor $D:\SW^f\functor\SW^f$ with a natural isomorphism $\textup{Id}_{\SW^f}\cong D\circ D$ that satisfies the property \beq \label{SW}\SW^f(X\wedge Z, Y)\cong\SW^f(X, DZ\wedge Y).\eeq Here $\wedge$ denotes the smash product of spectra. This phenomenon in $\SW^f$ is called {\em Spanier--Whitehead duality}. Using it one can see that the category $\SW^f$ is self-dual and that the Cogenerating Hypothesis is actually equivalent to the Generating Hypothesis. 
In the first part of this article (Section \ref{NGH}) we show that the Na{\"i}ve Generating Hypothesis in noncommutative stable homotopy fails to hold; however, our result is \emph{not} applicable to Freyd's Generating Hypothesis in finite stable homotopy. More precisely, we show

\begin{Thm}
The canonical map $\Phi':\NSH(A,B)\map\Hom_{\pi_*(\CC)}(\pi_*(A),\pi_*(B))$ in $\NSH$ is not injective in general. 
\end{Thm}

\begin{Rem}
 Our arguments below exploit the fact that noncommutative stable homotopy of stable $C^*$-algebras agrees with their $\E$-theory naturally. Thus the above result should be viewed as a failure of the Generating Hypothesis in bivariant $\E$-theory.
\end{Rem}

Spanier--Whitehead duality is a peculiar property of finite stable homotopy with many interesting consequences. It is a natural question to ask whether noncommutative stable homotopy also possesses this property. Our answer to this question is

\begin{Thm}
The Spanier--Whitehead duality functor $D$ on $\SW^f$ does not extend to $\NSH$.
\end{Thm}

In the second part of the paper (Section \ref{RefGH}) we explain the deficiencies that the na{\"i}ve extrapolation of Freyd's Generating Hypothesis suffers from. The main problem is the {\em size} of $\NSH$ and we provide a modified formulation by restricting our attention to a suitable subcategory, denoted by $\NS$ (see Question \ref{ModGH} entitled {\em Matrix Generating Hypothesis in $\NS$}). This modified formulation also implies Freyd's Generating Hypothesis in finite stable homotopy. Our modification does not {\em correct} the failure of Spanier--Whitehead duality in the noncommutative setting. We believe that it is an intrinsic feature of noncommutative stable homotopy that needs no remedy.

\smallskip
\noindent
{\bf Acknowledgements.} The author wishes to thank A. M. Bohmann, B. Jacelon, and K. Strung for helpful discussions. The author is also very grateful to R. Bentmann and J. D. Christensen for their constructive feedback, which prompted the author to write up Section \ref{RefGH}. Finally the author expresses his gratitude to the anonymous referee(s) for the comments that helped improve the exposition.

\section{Na{\"i}ve Generating Hypothesis in $\NSH$} \label{NGH}
Let $\Csep$ denote the category of separable $C^*$-algebras and $*$-homomorphisms. Recall from \cite{ThomThesis} that there is a stable homotopy functor $\pi_*:\Csep\functor \Mod(\pi_*(\CC))$ in noncommutative topology, which factors through $\NSH$ giving rise to the canonical map \beq\Phi':\NSH(A,B)\map\Hom_{\pi_*(\CC)}(\pi_*(A),\pi_*(B)).\eeq Here $\Mod(\pi_*(\CC))$ denotes the category of right modules over the ring $\pi_*(\CC)$. 

\subsection{Failure of the injectivity of $\Phi'$}
It is known that on the category of stable $C^*$-algebras noncommutative stable homotopy agrees with bivariant $\E$-theory \cite{ConHig,DadAsymHom}, i.e., $\NSH(A,B)\cong\E_0(A,B)$ and $\pi_*(A)\cong\E_*(A)$. In addition, on the category of nuclear $C^*$-algebras one has $\E_0(A,B)\cong\KK_0(A,B)$ and $\E_*(A)\cong\K_*(A)$ (see, for instance, the paragraph preceeding Section 6 in \cite{DadAsymHom}). It follows from the Universal Coefficient Theorem in $\KK$-theory \cite{RosSch} that there is a natural short exact sequence of abelian groups in the bootstrap class
\beq\label{UCT}
0\map\Ext^1(\K_*(\Sigma A),\K_*(B))\map\KK_*(A,B)\map\Hom(\K_*(A),\K_*(B))\map 0,
\eeq which splits unnaturally.

\begin{thm} \label{inj}
The canonical map $\Phi':\NSH(A,B)\map\Hom_{\pi_*(\CC)}(\pi_*(A),\pi_*(B))$ is not injective in general.
\end{thm}

\begin{proof}
Let us choose judiciously $A=\C(X,x)\prot\cpt$ and $B=\C(Y,y)\prot\cpt$ ($\cpt$ being the $C^*$-algebra of compact operators) in such a manner that $\Ext^1(\K_*(\Sigma A),\K_*(B))$ is non-zero (with every abelian group in \eqref{UCT} finitely generated). This can be easily achieved by choosing finite pointed CW complexes $(X,x)$ and $(Y,y)$, such that their $\K$-theory groups contain non-zero torsion. Then $\NSH(A,B)\cong\KK_0(A,B)$ due to the stability and nuclearity of all the $C^*$-algebras in sight. One has the following commutative diagram

\beqn
\xymatrix{
\NSH(A,B)\ar[rr]^{\Phi'}\ar@{=}_{\cong}[d] && \Hom_{\pi_*(\CC)}(\pi_*(A),\pi_*(B))\ar@{^{(}->}[d]\\
\KK_0(A,B)\ar[rr] && \Hom(\K_*(A),\K_*(B)).
}
\eeqn Here the left vertical map is an isomorphism as argued before and the right vertical map is an injection. Now the map $\KK_0(A,B)\map\Hom(\K_*(A),\K_*(B))$ is not injective due to the non-vanishing of the $\Ext^1$-term in \eqref{UCT}, whence $\Phi'$ cannot be injective. 
\end{proof}

\subsection{Failure of Spanier--Whitehead duality}
The Spanier--Whitehead category is a tensor triangulated category under smash product of finite spectra. Thom proved in Theorem 3.3.7 of \cite{ThomThesis} that $\NSH$ is a tensor (under maximal $C^*$-tensor product $\prot$) triangulated category. The sphere spectrum $\SS$ (resp. $(\CC,0)$) turns out to be the tensor unit in $\SW^f$ (resp. $\NSH$).

\begin{lem} \label{vanish}
$\NSH(M_2(A),\CC)\cong 0$ for all $A\in\NSH$.
\end{lem}

\begin{proof}
It is known that every element in $\NSH(M_2(A),\CC)$ can be represented as the homotopy class of a $*$-homomorphism $f:\Sigma^{r} M_2(A)\map \Sigma^{r}\CC$ (use Lemma 3.3.1 of \cite{ThomThesis} and Proposition 16 of \cite{DadAsymHom}). Since $\Sigma^{r}\CC$ is commutative and $M_2(\CC)$ is simple, $f$ must be the zero morphism. Indeed, observe that $\Sigma^{r}M_2(A)\cong M_2(\C(S^{r},\star)\prot A)$, where $(S^{r},\star)$ is the pointed $r$-sphere. The kernel of the said homomorphism must be of form $M_2(I)$, where $I$ is a closed two-sided ideal of $\C(S^{r},\star)\prot A$. It follows that the map $M_2((\C(S^{r},\star)\prot A)/I)\map \Sigma^{r}\CC$ is an isomorphism onto its image. Since the codomain is commutative, the two-sided ideal $I$ must be all of $\C(S^{r},\star)\prot A$, from which the assertion follows.
\end{proof}

\begin{thm} \label{Ndual}
The Spanier--Whitehead duality functor $D$ on $\SW^f$ does not extend to $\NSH$.
\end{thm}

\begin{proof}
Suppose one could extend the functor $D$ to $\NSH$ satisfying the analogue of \eqref{SW} $$\NSH(A\prot B, C)\cong\NSH(A,DB\prot C).$$ Set $A= DM_2(\CC)$, $B=M_2(\CC)$ and $C=\CC$. Then there would be an induced isomorphism $$\NSH(D M_2(\CC)\prot M_2(\CC),\CC)\cong\NSH(D M_2(\CC), D M_2(\CC)).$$ Using the previous Lemma \ref{vanish}, we conclude that $$\NSH(D M_2(\CC) \prot M_2(\CC),\CC)\cong\NSH(M_2(D M_2(\CC)),\CC)\cong 0.$$ This would force the identity map on $D M_2(\CC)$ to be equal to $0$ in $\NSH$. This is not possible (see, for instance, Remark 3.4 of \cite{MyNSH}). 
\end{proof}

\begin{rem}
In (equivariant) bivariant $\K$-theory there are positive results in this direction \cite{SchSW}. The constructions in ibid. involve $\sigma$-$C^*$-algebras. 
\end{rem}




\section{Modified Matrix Generating Hypothesis} \label{RefGH}
We called the Generating Hypothesis in the previous section {\em na{\"i}ve} for the following reasons:

\begin{enumerate}
 \item \label{prob2} In $\SW^f$ the objects are finite pointed CW complexes, whereas in $\NSH$ they are arbitrary pointed compact metrizable spaces. The presence of finite matrix algebras in the noncommutative analogue of $\SW^f$ is non-negotiable, since they form natural building blocks of noncommutative CW complexes \cite{EilLorPed1,PedCW}. However, $\cpt$ can be expressed as $\dlim_n M_n(\CC)$ in $\Csep$, i.e., it can be regarded as a countable inverse limit of {\em noncommutative or fat points}. Thus $\cpt$ is not indispensable.
 
 \item \label{prob1} We did not take finite matrix algebras into consideration as test objects for the Generating Hypothesis, which would be the correct way to think about the Generating Hypothesis in this situation \cite{BohMay}.
 \end{enumerate}
 
 \medskip

Let us first address issue number \eqref{prob2}. We already observed that $\NSH$ is a tensor triangulated category under $\prot$, where the tensor structure generalizes the smash product of finite spectra.

\begin{defn} \label{NSW}
We define the category of noncommutative finite spectra, denoted by $\NS$, to be the smallest thick tensor triangulated subcategory of $\NSH$ generated by $M_n(\CC)$ for all $n\in\NN$. It is the noncommutative analogue of the Spanier--Whitehead category $\SW^f$.
\end{defn}

\begin{que}[Matrix Generating Hypothesis in $\NS$] \label{ModGH}
Let $f\in\NS(A,B)$ be a morphism, such that $\NS(M_n(\CC),f):\NS(M_n(\CC),(A,i))\map\NS(M_n(\CC),(B,i))$ is the zero morphism for all $n\in\NN$ and $i\in\ZZ$. Is $f$ itself the zero morphism in $\NS(A,B)$?
\end{que}

The above formulation uses $M_n(\CC)$ for all $n\in\NN$ as test objects and hence addresses issue number \eqref{prob1}. The full subcategory of $\NS$ consisting of commutative $C^*$-algebras is equivalent to the opposite category of $\SW^f$ via the functor $(X,x)\mapsto\C(X,x)$. Due to the self-duality of $\SW^f$ one can also view it sitting inside $\NS$ via the duality functor $D$.

\begin{prop}
The Matrix Generating Hypothesis in $\NS$ implies Freyd's Generating Hypothesis in $\SW^f$.
\end{prop}

\begin{proof}
Since the Cogenerating Hypothesis is equivalent to the Generating Hypothesis in $\SW^f$, it suffices to show that the former is a consequence of the Matrix Generating Hypothesis. Let $A=\C(X,x)$ and $B=\C(Y,y)$ be separable commutative $C^*$-algebras. Now suppose $f\in\NS(A,B)\cong\SW^f((Y,y),(X,x))$ is a morphism that induces the zero morphism $\pi_*(f):\pi_*(A)\cong\pi^*(X,x)\map\pi^*(Y,y)\cong\pi_*(B)$. We need to show that $f$ itself is the zero morphism in $\NS(A,B)\cong\SW^f((Y,y),(X,x))$. Since $B$ is commutative, arguing as in Lemma \ref{vanish} we deduce that $\NS(M_n(\CC),(B,i))=0$ for all $n>1$ and $i\in\ZZ$. It follows that $\NS(M_n(\CC),f)=0$ for all $n\in\NN$ and $i\in\ZZ$, whence by the Matrix Generating Hypothesis $f=0$.
\end{proof}

\begin{rem}
Notice that Proposition \ref{inj} is no longer applicable to $\NS$. Indeed, the $C^*$-algebra $\cpt$ does not belong to $\NS$ anymore as it is not finitely built. If we were to include $\cpt$ in the generating set of Definition \ref{NSW}, then the corresponding predictable {\em Matrix + Compact Generating Hypothesis} would once again be falsified by Proposition \ref{inj}.
\end{rem}

Let us reiterate that our results in Section \ref{NGH} do not invalidate Freyd's Generating Hypothesis in finite stable homotopy.  The failure of the injectivity of $\Phi'$ in the bigger category $\NSH$ crucially exploited the presence of genuinely noncommutative $C^*$-algebras. We hope that the generalized perspective will shed some light on the original problem.


\bibliographystyle{abbrv}
\bibliography{/home/ibatu/Professional/math/MasterBib/bibliography}

\smallskip

\end{document}